\documentclass[a4paper,11pt]{amsart}

\usepackage{amsmath,amssymb,amsthm,amsfonts}
\usepackage{mathtools}
\usepackage{xcolor}
\usepackage{tikz}
\usetikzlibrary{shapes} 
\usetikzlibrary{decorations.pathmorphing} 
\usetikzlibrary{cd}
\usepackage{ifthen}

\usepackage{ulem}
\usepackage{enumitem}
\usepackage{mdwlist} 
\usepackage{booktabs}
\usepackage{hyperref}
\hypersetup{
    colorlinks,
    linkcolor={red!50!black},
    citecolor={blue!50!black},
    urlcolor={blue!80!black}
}
\usepackage{calc}
\usepackage{aliascnt} 
\usepackage{amscd}
\usepackage{dutchcal} 

\newcommand{\uA}{\underline A}
\newcommand{\eps}{\varepsilon}

\newcommand{\TTT}{\mathsf{T}\!}  
\newcommand{\PPP}{\mathsf{P}\!}

\newcommand{\ds}[1]{(#1)} 

\newcommand{\Db}{\mathsf{D}^b}  
\newcommand{\cZ}{\mathcal{Z}}
\newcommand{\cA}{\mathcal{A}}

\newcommand{\cE}{\mathcal{E}}
\newcommand{\cT}{\mathcal{T}}

\newcommand{\cH}{\mathcal{H}}
\newcommand{\cP}{\mathcal{P}}

\newcommand{\EE}{\mathcal{E}}

\newcommand{\LL}{\mathcal{L}}

\newcommand{\OO}{\mathcal{O}}

\newcommand{\IC}{\mathbb{C}}
\newcommand{\IP}{\mathbb{P}}

\newcommand{\IZ}{\mathbb{Z}}
\newcommand{\IN}{\mathbb{N}}

\newcommand{\kk}{\mathbf{k}}

\newcommand{\HH}{\mathsf{HH}} 

\newcommand{\abrunden}[1]{{\left\lfloor #1 \right\rfloor}}
\newcommand{\aufrunden}[1]{{\left\lceil #1 \right\rceil}}

\newcommand{\genby}[1]{{\langle #1 \rangle}}
\newcommand{\Genby}[1]{{\langle\!\langle #1 \rangle\!\rangle}}

\newcommand{\medoplus}[1]{\underset{#1}{\raisebox{.2ex}{$\bigoplus$}}}
\newcommand{\medwedge}[1]{\overset{#1}{\raisebox{.2ex}{$\bigwedge$}}}

\newcommand{\op}{^\text{op}}

\newcommand{\blank}{\:\cdot\:} 

\newcommand{\eqqcolon}{=\mathrel{\mathop:}}
\newcommand{\linearised}[1]{\{#1\}} 
\newcommand{\poslin}[1]{^{\linearised{#1}}}
\newcommand{\poslinn}{\poslin{n}}
\newcommand{\neglin}[1]{^{-\linearised{#1}}}
\newcommand{\neglinn}{\neglin{n}}

\newcommand{\reg}{\mathcal O}

\DeclareMathOperator{\D}{\mathsf D}

\DeclareMathOperator{\QCoh}{QCoh}
\DeclareMathOperator{\Coh}{Coh}

\DeclareMathOperator{\Ho}{\mathsf H}
\mathchardef\mhyphen="2D
\newcommand{\hh}{\,\mhyphen}

\DeclareMathOperator{\RHom}{RHom}
\DeclareMathOperator{\Hom}{\mathsf{Hom}}
\DeclareMathOperator{\End}{\mathsf{End}}
\DeclareMathOperator{\Aut}{Aut}
\DeclareMathOperator{\Ext}{\mathsf{Ext}}
\DeclareMathOperator{\Tor}{\mathsf{Tor}}
\DeclareMathOperator{\maxdeg}{\mathsf{maxdeg}}
\DeclareMathOperator{\mindeg}{\mathsf{mindeg}}

\DeclareMathOperator{\id}{id}

\DeclareMathOperator{\Pic}{Pic}

\DeclareMathOperator{\Spec}{Spec}
\DeclareMathOperator{\Bl}{Bl}
\DeclareMathOperator{\Cone}{Cone}

\DeclareMathOperator{\PT}{\mathsf{P}}

\DeclareMathOperator{\Dgmod}{\mathsf{dg-Mod}}
\DeclareMathOperator{\Mod}{\mathsf{Mod}}
\let\mod\relax
\DeclareMathOperator{\mod}{\mathsf{mod}}
\let\deg\relax
\DeclareMathOperator{\deg}{\mathsf{deg}}
\let\min\relax
\DeclareMathOperator{\min}{\mathsf{min}}
\let\max\relax
\DeclareMathOperator{\max}{\mathsf{max}}

\newcommand{\Xan}{X^{\mathsf{an}}}

\newcommand{\isom}{ \text{{\hspace{0.48em}\raisebox{0.8ex}{${\scriptscriptstyle\sim}$}}}
                    \hspace{-0.65em}{\rightarrow}\hspace{0.3em}} 

\newcommand{\into}{\hookrightarrow}
\newcommand{\onto}{\twoheadrightarrow}

\newcommand{\xto}[1]{\xrightarrow{#1}}

\newcommand{\arXiv}[1]{\href{http://arxiv.org/abs/#1}{\texttt{arXiv:#1}}}
\newcommand{\jstor}[1]{{\href{http://www.jstor.org/stable/#1}{JSTOR:#1}}}

\newcommand{\bib}[5]{{\bibitem{#1} #2, {\emph{#3},} #4#5.}}
\newcommand{\hyref}[2]{\hyperref[#2]{#1~\ref*{#2}}}

\newcommand{\sym}{\mathfrak S}
\newcommand{\alt}{\mathfrak a}
\newcommand{\fm}{\mathfrak m}

\newtheorem{theorem}{Theorem}[section]

\newtheorem*{theorem*}{Theorem}
\newtheorem*{conjecture}{Conjecture}
 
\newtheorem{theoremalpha}{Theorem}

  \newaliascnt{proposition}{theorem}
  \newtheorem{proposition}[proposition]{Proposition}
  \aliascntresetthe{proposition}
  
  \newaliascnt{lemma}{theorem}
  \newtheorem{lemma}[lemma]{Lemma}
  \aliascntresetthe{lemma}
  
  \newaliascnt{corollary}{theorem}
  \newtheorem{corollary}[corollary]{Corollary}
  \aliascntresetthe{corollary}

\theoremstyle{definition}
  \newaliascnt{definition}{theorem}
  \newtheorem{definition}[definition]{Definition}
  \aliascntresetthe{definition}
  
  \newaliascnt{remark}{theorem}
  \newtheorem{remark}[remark]{Remark}
  \aliascntresetthe{remark}
  
  \newaliascnt{question}{theorem}
  
  \aliascntresetthe{question}
  
  \newaliascnt{example}{theorem}
  \newtheorem{example}[example]{Example}
  \aliascntresetthe{example}
  
\begin{document}

\title[Formality of $\IP$-objects]{Formality of $\IP$-objects}
\author[A. Hochenegger]{Andreas Hochenegger}
\author[A. Krug]{Andreas Krug}

\maketitle

\begin{abstract}
We show that a $\IP$-object and simple configurations of $\IP$-objects have a formal derived endomorphism algebra. Hence the triangulated category (classically) generated by such objects is independent of the ambient triangulated category. We also observe that the category generated by the structure sheaf of a smooth projective variety over the complex numbers only depends on its graded cohomology algebra.
\end{abstract}

\begingroup\renewcommand\thefootnote{}%
\addtocounter{footnote}{-2}\endgroup

\tableofcontents

\addtocontents{toc}{\protect\setcounter{tocdepth}{-1}}  
\section*{Introduction}

In recent decades, triangulated categories have become very popular in representation theory and algebraic geometry. 
Given an object $E$ in a $\kk$-linear triangulated category $\cT$, one can consider the triangulated subcategory generated by $E$ inside $\cT$.
Its complexity depends strongly on the graded endomorphism algebra $\End^*(E) = \bigoplus_{i\in\IZ} \Hom(E,E[i])$. 

For example, let $E\in \cT$ be an exceptional object, that is,  $\End^*(E)=\kk$. 
In this case, $E$ generates a category equivalent to the derived category of vector spaces, which can be regarded as the smallest and simplest $\kk$-linear triangulated category. 

In general, due to a result by Keller \cite{Keller-derivingdg}, the generated category $\genby E$ can always be identified with the derived category $\D(B)$ of some differential graded (dg) algebra $B$ whose graded cohomology algebra coincides with the graded endomorphism algebra: $\Ho^*(B)\cong \End^*(E)$. 
(Depending on the exact definition of the category generated by one object, we may have to replace the derived category $\D(B)$ by its subcategory of compact objects, but we will ignore this issue in this introduction and return to it in \autoref{sec:formality-triangulated}.)

Of course, the situation is most pleasant if we already have 
\begin{align}
\tag{$\ast$}
\label{eq:goodsituation}
\genby E = \D(\End^*(E))
\end{align}
 so that the generated category only depends on the graded endomorphism algebra but not on the ambient category $\cT$. In this paper, we provide two situations in which \eqref{eq:goodsituation} holds: For $E$ the direct sum of $\IP$-objects that form a tree (in particular, if $E$ is a single $\IP$-object) and for $E=\reg_X$ the structure sheaf of a smooth projective variety.

It follows from Keller's result that a sufficient condition for \eqref{eq:goodsituation} to hold is that the graded algebra $A \coloneqq \End^*(E)$ is \emph{intrinsically formal}. 
This means that every dg-algebra $B$ with $\Ho^*(B)=A$ is actually quasi-isomorphic to $A$. 
A very useful sufficient criterion for intrinsic formality in terms of vanishing of Hochschild cohomology was given by Kadeishvili \cite{Kadeishvili}. 
This was used by Seidel and Thomas \cite{Seidel-Thomas} to prove intrinsic formality of the endomorphism algebra of $A_n$-configurations of spherical objects. 
The endomorphism algebra of a single spherical object is of the form $\End^*(E)=\kk\oplus \kk[-d]$ for some $d\in \IZ$. 
So, spherical objects are arguably the second simplest type of objects in triangulated categories after exceptional objects. 
Besides this, the main reason for interest in spherical objects is the fact that they induce autoequivalences, so-called \emph{spherical twists}, of triangulated categories. 
We also want to mention that Keller, Yang and Zhou studied the Hall algebra of a triangulated category generated by a single spherical object in \cite{Keller-Yang-Zhou}.

The notion of spherical objects was generalised by Huybrechts and Thomas \cite{Huybrechts-Thomas} to that of $\IP$-objects. 
These objects again induce twist autoequivalences. Furthermore, they play an important role in the theory of hyperk\"ahler manifolds,
but appear also in symplectic geometry; see, for example, \cite{Mak-Wu}.
The graded endomorphism algebra of a $\IP$-object is still rather simple: namely it is generated by one element. 
More precisely, for $n,k$ positive integers, an object $P\in \cT$ is called a \emph{$\IP^n[k]$-like object} if 
\[
 \End^*(P)=\kk[t]/t^{n+1} \quad\text{with}\quad \deg(t)=k\,.
\]
Such an object is called a \emph{$\IP^n[k]$-object} (or just \emph{$\IP$-object}) if it is additionally a Calabi--Yau object; see \autoref{def:P} for details.

\begin{theoremalpha}
\label{main:pn-single}
Let $P$ be a $\IP^n[k]$-like object with $k\ge 2$.
Then $\End^*(P)$ is intrinsically formal so that $\genby{P} \cong \D(\End^*(P))$ is independent of the ambient triangulated category.
\end{theoremalpha}

One application is that the associated $\IP$-twist can be written as the twist along a spherical functor $F\colon \D(\kk[t])\to \cT$; see \autoref{cor:pntwist-spherical-functor}. 
This is actually a result due to Segal \cite[Prop.~4.2]{Segal}; we provide that the formality assumption there is always given.

A \emph{tree} of $\IP^n[k]$-objects in a triangulated category $\cT$ is given by a  
collection of $\IP^n[k]$-objects $P_i\in \cT$, one for every vertex of a connected graph without loops, such that $\dim_{\kk}\Hom^*(P_i, P_j)=1$ if $i$ and $j$ are adjacent in the graph and $\Hom^*(P_i, P_j)=0$ otherwise.

\begin{theoremalpha}
\label{main:pn-many}
Let $\{P_1,\ldots,P_m\}$ be a tree of $\IP^n[k]$-objects with either $n$ even and $k \geq 2$ or $n=1$ and $k \geq 4$ (the spherical case).
Then 
$\genby{P_1,\ldots,P_m} \cong \D(\End^*(\bigoplus_i P_i))$ 
is independent of the ambient triangulated category.
\end{theoremalpha}

Our proof uses Kadeishvili's criterion for intrinsic formality together with a description of minimal resolutions of graded algebras due to Butler and King \cite{Butler-King}.

\autoref{main:pn-many} might be useful in order to prove a faithfulness result for actions induced by $A_m$-configurations of $\IP$-objects; see \autoref{subsec:faithful} for some more explanation on this.

Let $X$ be a smooth projective variety over $\IC$.
A distinguished object in its bounded derived category of coherent sheaves $\Db(\Coh(X))$ is given by the structure sheaf $\reg_X$. 
We use the formality of the Dolbeault complex to prove the following result.

\begin{theoremalpha}
\label{main:structuresheaf}
The category generated by $\reg_X$ in $\Db(\Coh(X))$ only depends on the graded algebra $\End^*(\reg_X)\cong \Ho^*(\reg_X)$. More precisely,
\[
 \genby{\reg_X}\cong \D(\Ho^*(\reg_X))\,.
\]
\end{theoremalpha}

\noindent
This result may be of interest for the conjecture that the graded algebra $\Ho^*(\reg_X)$ is a derived invariant of smooth projective varieties. 

This paper is organised as follows. In \autoref{sec:preliminaries} we fix the notation and collect well-known facts on triangulated categories and dg-algebras. In \autoref{sec:P} we recall the definition of $\IP$-objects and their associated twist. Then, in \autoref{sec:pn-single}, we prove that $\genby P \cong \D(\End^*(P))$ for a $\IP$-object $P$. 
We review the description of the terms of minimal resolutions of graded algebras due to Butler and King \cite{Butler-King}, in the following \autoref{sec:min}. We go through the main steps of its proof to make sure that the results hold in our graded setting. In \autoref{sec:pobjects}, we prove \autoref{main:pn-many} using the description of the terms of the minimal resolutions in order to obtain the vanishing of the relevant Hochschild cohomology. Actually our results on configurations of $\IP^n[k]$-like objects are more general than stated above; see \autoref{prop:pn-moregeneral} and \ref{prop:spherical-moregeneral}. 
We prove \autoref{main:structuresheaf} in \autoref{sec:OX}, 
actually for compact complex manifolds satisfying the $\partial\bar\partial$-Lemma.
In the final \autoref{sec:example} we give a general construction which produces trees of $\IP$-objects out of trees of spherical objects. As a geometric application, we explicitly construct trees of $\IP$-objects on Hilbert schemes of points on surfaces; see \autoref{app:geometric-pns}.

\subsection*{Acknowledgements}

We thank Daniel Huybrechts for asking about the formality of $\IP$-objects and for helpful comments.
Moreover, we are grateful for comments and suggestions from
Ben Anthes, Elena Martinengo, S\"onke Rollenske, Theo Raedschelders, Paolo Stellari, Greg Stevenson, Olaf Schn\"urer and an anonymous referee.
Finally, we mention that Gufang Zhao had already obtained a partial result on the formality of a single $\IP^n$-object with $n \leq 4$.

\addtocontents{toc}{\protect\setcounter{tocdepth}{1}}   
\section{Triangulated categories, dg-algebras and Hochschild cohomology}
\label{sec:preliminaries}

\subsection{Conventions on algebras}

The letter $\kk$ will denote an algebraically closed field. 
All our algebras $A$ will be $\kk$-algebras, and whenever we speak of a graded algebra we mean a graded $\kk$-algebra $A = \bigoplus_{i\in \IZ} A^i$.

For a (graded) $\kk$-algebra $A$, we denote by the (graded) tensor product $A^e \coloneqq A \otimes_\kk A^{op}$ its \emph{enveloping algebra}.
Whenever we speak in the following about ideals or modules over some (not necessarily commutative) algebra, 
we refer to finitely generated left ideals or modules. The formalism of enveloping algebras allows us to speak of $A\hh A$-bimodules as left $A^e$-modules.

\subsection{Conventions on triangulated categories}
\label{sec:conv-tricat}

All triangulated categories are assumed to be $\kk$-linear, and subcategories thereof to be triangulated and full.
The shift functor will be denoted by $[1]$.
All triangles are meant to be distinguished and denoted by $A\to B\to C$, hiding the morphism $C \to A[1]$.
We write $\Hom^*(A,B) = \bigoplus_{i\in\IZ} \Hom(A,B[i])[-i]$ for the derived homomorphisms in a triangulated category: this is a complex equipped with the zero differential.
In contrast, $\Hom^\bullet(A,B)$ will be the homomorphism complex if $A$ and $B$ are objects of a dg-category, usually an enhancement of a triangulated category. In the literature this is sometimes also denoted by $\RHom(A,B)$.

All functors between triangulated categories are meant to be exact. In particular, we will abuse notation and write $\otimes$ for the derived functor $\otimes^L$, using the same symbol as for the functor between abelian categories.

\subsection{dg-algebras and Hochschild cohomology}
\label{sec:hochschild}

\begin{definition}
An \emph{dg-algebra} $A$ (over $\kk$) consists of a graded $\kk$-vector space
\[
A = \bigoplus_{i\in \IZ} A^i
\]
and graded $\kk$-linear maps 
\begin{itemize}
\item $d \colon A \to A$ of degree $1$ and
\item $m \colon A \otimes A \to A$ of degree $0$
\end{itemize}
satisfying the following compatibilities:
\begin{itemize}
\item $d^2=0$, so $d$ is a \emph{differential};
\item $m(m\otimes\id) = m(\id\otimes m)$, so $m$ is an associative multiplication;
\item $d(m(a,b)) = m(da,b) + (-1)^{\deg a} m(a,db)$ for homogeneous elements (the Leibniz rule).
\end{itemize}
\end{definition}

\begin{example}
A graded algebra is a dg-algebra with $d=0$ and $m(a,b) = a\cdot b$ its multiplication. 
Let $A$ be a dg-algebra. Then $m$ induces a multiplication on the cohomology $\Ho(A):=\oplus_{i\in \IZ}A^i$ of $A$ with respect to $d$. So $\Ho(A)$ is a graded algebra.
\end{example}

A morphism $\phi\colon A \to B$ of dg-algebras is a $\kk$-linear map compatible with differential and multiplication, that is, $\phi(d_Aa) = d_B(\phi(a))$ and $\phi(m_A(a,b)) = m_B(\phi(a),\phi(b))$.
Note that $\phi$ induces a map on cohomology $\Ho(\phi)\colon \Ho(A) \to \Ho(B)$.

\begin{definition}
\label{def:quasi-isomorphism}
Let $\phi\colon A \to B$ be a morphism of dg-algebras. 
Then $\phi$ is called a \emph{quasi-isomorphism} if $\Ho(\phi)\colon \Ho(A) \to \Ho(B)$ is an isomorphism of graded $\kk$-algebras.

We say that two dg-algebras are \emph{quasi-isomorphic} if they can be connected by a finite zigzag of quasi-isomorphisms.
\end{definition}

\begin{definition}
A graded algebra $A$ is called \emph{intrinsically formal} if any two dg-algebras with cohomology $A$ are quasi-isomorphic;
or equivalently, if any dg-algebra $B$ with $\Ho(B)=A$ is already quasi-isomorphic to $A$.
\end{definition}

Recall that we denote by $A^e = A \otimes_\kk A\op$ the enveloping algebra of $A$. Note that $A$ has a natural $A^e$-module structure by multiplication from left and right. Given $q\in \IZ$ and two graded $A^e$-modules $N$ and $M$, we write $\Hom_{A^e}^q(N,M)$ for the $A^e$-module homomorphisms which are homogeneous of degree $q$. 

\begin{definition}\label{def:bar}
The complex 
\[
B^\bullet\colon \cdots \to  A^{\otimes (q+2)} \xto{d^{-q}} A^{\otimes (q+1)} \to \cdots \to A^{\otimes 2} \to 0, 
\]
with $A^{\otimes(q+2)}$ in degree $q$, is called the \emph{Bar resolution} of $A$ as an $A^e$-module, where $d = d^{-q}$ is given by
\[
d(a_1 \otimes \cdots \otimes a_{q+2}) = \sum_i \pm a_1 \otimes \cdots \otimes a_i \cdot a_{i+1} \otimes \cdots \otimes a_{q+2}.
\]
Note that the differentials are of degree zero.
\end{definition}

\begin{definition}
Let $A$ be a graded algebra and $M$ a graded $A^e$-module.
The \emph{Hochschild cohomology} of $A$ with values in $M$ is given by 
\[
\HH^{p,q}(A,M) \coloneqq \Ho^p( \Hom_{A^e}^q(B^\bullet,M)) \quad\text{for $p,q\in \IZ$.}
\]
\end{definition}

Our main tool for proving intrinsic formality of certain graded algebras will be the following result due to Kadeishvili.

\begin{proposition}[{\cite[Cor.~4]{Kadeishvili}, cf. \cite[Thm.~4.7]{Seidel-Thomas}, \cite[Cor.~1.9]{Roitz-White}}]
\label{prop:intrinsically-formal}
Let $A$ be a graded algebra. If $\HH^{q,2-q}(A,A)$ vanishes for $q>2$, then $A$ is intrinsically formal.
\end{proposition}

\begin{definition}
For an $A^e$-module $M = \bigoplus M^q$, its \emph{shift in degree} by $i$ is the $A^e$-module
$M\ds{i}$ with $M\ds{i}^q = M^{q+i}$. 
\end{definition}

\begin{remark}\label{rem:HHres}
Note that $\Hom_{A^e}^q(A,A)=\Hom_{A^e}^0(A,A(q))$. Moreover, graded $A^e$-modules, with $\Hom_{A^e}^0$ as morphisms, form an abelian category $\textsf{gr}A^e$. There, the Bar resolution is a projective resolution of the $A^e$-module $A$.  
It follows that $\HH^{p,q}(A,M) = \Ext^p_{\textsf{gr}A^e}(A,M(q))$. Hence, for any other  projective resolution $P^\bullet$ of $A$ as a graded $A^e$-module, we also have
\[
\HH^{p,q}(A,M)=\Ho^p\bigl(\Hom_{A^e}^q(P^\bullet, M)\bigr)= \Ho^p\bigl(\Hom_{A^e}^0(P^\bullet, M(q))\bigr).
\]
\end{remark}

\subsection{Derived categories and dg-categories}

In this paper we will encounter two types of derived categories. First, for an abelian category $\cA$ there is the category $\D(\cA)$ which is the localisation of the homotopy category of complexes with values in $\cA$ at the class of quasi-isomorphisms.
In our examples, the abelian category $\cA$ will be a category of (coherent or quasi-coherent) $\reg_X$-modules over a variety or manifold $X$. For details on $\D(\cA)$ see, for example, \cite[Ch.~2]{Huybrechts}.

Let us very quickly recall some facts and fix notation concerning dg-categories and enhancements; see, for example, \cite[\S 3]{Kuz-Lunts} for details. 
A \emph{dg-category} is a $\kk$-linear category $\cE$ whose Hom-spaces are dg-$\kk$-modules and the compositions are compatible with the dg-structure.
The \emph{homotopy category} $\Ho^0(\cE)$ is defined to have the same objects as $\cE$ and morphisms $\Hom_{\Ho^0(\cE)}(E, F):=\Ho^0(\Hom_{\cE}(E,F))$.
The category $\Dgmod(\cE)$ of (right) dg-modules over $\cE$ is defined as the category of dg-functors from $\cE\op$ to the category of dg-modules over $\kk$. Its homotopy category $\Ho^0(\Dgmod(\cE))$ carries the structure of a triangulated category. The dg-category $\cE$ is called \emph{pretriangulated} if the image of the Yoneda embedding $\Ho^0(\cE)\hookrightarrow \Ho^0(\Dgmod(\cE))$ is a triangulated subcategory.         

Given a triangulated category $\cT$, a \emph{dg-enhancement} of $\cT$ is a pretriangulated dg-category $\cE$ together with an exact equivalence $\Phi\colon \Ho^0(\EE) \isom \cT$. 

We can consider every dg-algebra $A$ as a dg-category with one object. Then the homotopy category $\Ho^0(\Dgmod(\cE))$ of dg-modules over that category agrees with the usual notion of the category of dg-modules over the algebra $A$. 
The derived category of the a dg-algebra $A$ is defined as $\D(A):=\Ho^0(\Dgmod(A))[\mathsf{qis}^{-1}]$, the category $\Dgmod(A)$ of right dg-modules over $A$ localised at the class of quasi-isomorphisms; for details see, for example, \cite[\S 8]{Keller-tilting}. 
If the dg-algebra $A$ is concentrated in degree 0 (i.e.\ it is an ordinary algebra), then $\D(A)=\D(\Mod(A))$ where the latter is the derived category of the abelian category $\Mod(A)$ in the sense explained above.

\subsection{Formality and triangulated categories}
\label{sec:formality-triangulated}

The following results are well known to experts and can be found in essence or in part in, for example, the survey \cite{Keller-dg} by Keller, the lecture notes \cite{Toen} by To\"en or the book \cite[\S 10]{Bernstein-Lunts} by  Bernstein and Lunts. 

We recall terminology. Let $\cT$ be a triangulated category.  
The category $\cT$ is called \emph{cocomplete} if arbitrary direct summands exist. 
It is called \textit{idempotent complete} if every idempotent endomorphism splits. An object $T\in \cT$ is called \emph{compact} if for every set $\{Y_i\}$ of objects in $\cT$ the natural morphism
\[
\medoplus{i} \Hom(T, Y_i)\to \Hom(T, \medoplus{i} Y_i)
\]
is an isomorphism. We write $\cT^c\subset \cT$ for the full subcategory of compact objects. It is a \emph{thick} (i.e.\ closed under direct summands) triangulated subcategory of $\cT$.

For an object $T\in \cT$, we write $\Genby{T}$ for the smallest cocomplete triangulated subcategory of $\cT$ containing $T$. 
We write $\genby{T}$ for the smallest thick  triangulated subcategory of $\cT$ containing $T$. If $T$ is compact, then  $\genby{T}\subset \cT^c$. 
There is an inclusion $\genby{T} \subset \Genby{T}$, since every cocomplete triangulated category is thick; see \cite[Prop.\ 1.6.8]{Neeman-book}.

The compact objects in the derived category $\D(\QCoh(X))$ of coherent sheaves on a separated scheme of finite type over $\kk$ coincide with the \emph{perfect} objects, that is, objects locally quasi-isomorphic to bounded complexes of locally free sheaves of finite rank.

The following statement can be found in a similar form in \cite[\S 4.2]{Keller-derivingdg} or \cite[Prop.\ 1.16 \& 1.17]{Lunts-Orlov}.

\begin{theorem}
\label{thm:generatedcat}
Let $\cT$ be a triangulated category with a dg-en\-han\-ce\-ment given by a dg-category $\EE$ and an equivalence $\Phi\colon \Ho^0(\EE) \isom \cT$.
Let $T\in \cT$ be a compact object, $E\in \EE$ some object with $\Phi(E)\cong T$, and consider the dg-algebra $B=\Hom^\bullet(E,E)$. 
\begin{enumerate}
\item\label{thm:i} Let $\cT$ be idempotent complete. Then there is an exact equivalence $\genby{T}\cong \D(B)^c$.
\item Let $\cT$ be cocomplete. Then there are exact equivalences 
\[
\Genby{T} \cong \D(B)\quad \text{and} \quad \genby{T} \cong \D(B)^c\,. 
\]
\end{enumerate}
\end{theorem}

\begin{proof}
This follows from \cite[Prop.\ B.1]{Lunts-Schn} by plugging in $P=I=E$ and $z=\id_E$. 
\end{proof}

\begin{remark}
The first part of the theorem applies to any bounded derived category $\Db(\cA)$ of an abelian category $\cA$. On the one hand, $\Db(\cA)$ admits a (not necessarily unique) dg-enhancement. On the other hand, $\Db(\cA)$ is automatically idempotent complete, which holds even if $\cA$ is just an idempotent complete exact category by \cite[Thm.~2.8]{Balmer-Schlichting}.
\end{remark}

\begin{corollary}
\label{cor:intformal}
Let $\cT$ be a cocomplete dg-enhanced triangulated category and $T\in \cT^c$ a compact object.
If the graded algebra $A = \End^*(T)$ is intrinsically formal, then 
\[
\Genby{T}\cong \D(A) \quad\text{and}\quad \genby{T}\cong \D(A)^c\,. 
\]
\end{corollary}

\begin{proof}
Let $E \in \EE$ be some object with $\Phi(E)=T$ and $B = \Hom^\bullet(E,E)$.
Then $\Genby{T} \cong \D(B)$ by \autoref{thm:generatedcat}.
As by assumption $A$ is intrinsically formal, $\D(B) \cong \D(A) = \D(\Mod(A))$ by \cite[Thm.~10.12.5.1]{Bernstein-Lunts}.
Restricting to compact objects, this yields that $\genby{T} = \D(B)^c = \D(\Mod(A))^c$. 
\end{proof}

\begin{remark}
Let $A$ be an intrinsically formal graded algebra. By the corollary above, the categories generated by objects $T\in \cT$ with $\End^*(D)=A$ are all equivalent. In particular, the category generated by such an object is independent of the ambient cocomplete dg-enhanced triangulated category~$\cT$.
\end{remark}

\section{$\IP$-objects}\label{sec:P}

\subsection{Definition and basic examples}

\begin{definition}\label{def:P}
Let $P$ be an object in a $\kk$-linear triangulated category $\cT$.
\begin{itemize}
\item If $\End^*(P) \cong \kk[t]/t^{n+1}$ as graded $\kk$-algebras with $\deg(t)=k$, then we call $P$ a \emph{$\IP^n[k]$-like object}.
\item If a $\IP^n[k]$-like object $P$ is also $nk$-Calabi--Yau object, i.e.\ $\Hom^*(P,\blank) = \Hom^*(\blank,P[nk])^\vee$ functorially, then $P$ is called a \emph{$\IP^n[k]$-object}.
\end{itemize}
In some cases we will omit the integers $n,k$ and just speak of $\IP$-like objects or $\IP$-objects.
\end{definition}

$\IP^1[k]$-objects are well-known as \emph{spherical objects}; see \cite{Seidel-Thomas}. Without the Calabi--Yau property they are called \emph{spherelike objects} and studied in \cite{Hochenegger-Kalck-Ploog,Hochenegger-Kalck-Ploog-RT} by Kalck, Ploog and the first author.

$\IP^n[2]$-objects are known as \emph{$\IP^n$-objects} and studied in \cite{Huybrechts-Thomas} by Huybrechts and Thomas. The focus there is on hyperk\"ahler manifolds, whose structure sheaves are $\IP^n$-objects.
Another standard example of a $\IP^n$-object is the structure sheaf $\reg_Z\in \D(X)$ of the centre $\reg_{\IP^n}\cong Z\subset X$ of a Mukai flop of a variety of dimension $\dim X=2n$. 

The terminology $\IP^n[k]$ was introduced by the second author in \cite{Krug}, where examples of varieties are also given, whose structure sheaves are $\IP^n[k]$-objects.

\begin{remark}
A $\IP$-like object $P$ is already a $\IP$-object in $\genby{P}$.
Therefore our main question about the independence of $\genby{P}$ of the ambient category does not rely on the Calabi--Yau property.
As a (possibly misleading) consequence, in \cite{Keller-Yang-Zhou} the Calabi--Yau property of spherical objects is never mentioned.
\end{remark}

\begin{remark}
Let $X$ be a variety of dimension $nk$ such that $\OO_X$ is a $\IP^n[k]$-like object in $\Db(X)$, that is, $\End^*(\OO_X) = \kk[t]/t^{n+1}$ and $\deg(t)=k$.
Note that $\End^*(\OO_X) = H^*(\OO_X)$ as graded $\kk$-algebras, where the Yoneda product on the left becomes the cup product on the right.
As the cup product is graded commutative, $k$ odd implies immediately $t^2=0$, so $n=1$ and $\OO_X$ is spherelike.
Consequently, $n>1$ is only possible for even $k$.

However, the graded endomorphism algebra $\End^*(E)$ of an arbitrary object $E\in \cT$ does not need to be graded commutative. In fact, there are examples of $\IP^n[k]$-like objects with $n\ge 2$ and $k$ odd.
For a trivial example, consider the dg-algebra $A = \kk[t]/t^{n+1}$ with trivial differential and $\deg(t)=k$, where $n \ge 0$ and $k$ are integers. Then $A$ is a $\IP^n[k]$-object inside $\D(A)$.
For examples of $\IP^n[1]$-objects of geometric origin, see \cite[Ex. 4.2 (5) \& (6)]{Addington}.
\end{remark}

\subsection{Associated $\IP$-twists}
\label{sec:p-twists}

In this subsection we assume that $\cT$ is a $\kk$-linear triangulated category that admits a dg-enhancement and that the $\IP^n[k]$-object $P\in \cT$ is \emph{proper}, that is, $\Hom^*(P,F)$ is a finite-dimensional graded vector space for all $F \in \cT$. 
Under these assumptions there is an autoequivalence $\PPP_P\colon \cT\isom \cT$, the \emph{$\IP$-twist along $P$}, whose construction, due to \cite{Huybrechts-Thomas}, we sketch in the following.
Whenever we speak about the $\IP$-twist associated to a $\IP$-like object in later sections, we will tacitly assume that these assumptions are met.

So, let $P$ be a $\IP^n[k]$-object and $t$ be a non-zero element of $\Ext^k(P,P)$.
Using this generator, one can define the upper triangle for any $F \in \cT$:
\[
\begin{tikzcd}
\Hom^*(P,F)\otimes P[-k] \ar[r, "H", "t\otimes \id - \id \otimes t"'] \ar[dr, "\mathrm{ev} \circ H"'] &[3em] \Hom^*(P,F) \otimes P \ar[r] \ar[d, "\mathrm{ev}"] & \Cone(H) \ar[ld, dashrightarrow] \\
&  F  \ar[dl, dashrightarrow]\\
\PPP_P(F)
\end{tikzcd}
\]
As the composition $\mathrm{ev}\circ H = 0$, the arrow $\Cone(H) \dashrightarrow F$ exists. Completing this arrow to a triangle gives the double cone which we denote by $\PPP_P(F)$. 

\begin{remark}
The dg-enhancement of $\cT$ is necessary to actually define the $\IP$-twist $\PPP_P$ as a functor. 
In the case of spherical twists see \cite{Anno-Logvinenko} for a proper treatment and \cite[\S 3.1]{Hochenegger-Kalck-Ploog} for a rough idea.
In the geometric setting, Fourier-Mukai kernels allow us to circumvent dg-categories; see \cite[\S 2]{Huybrechts-Thomas}, where it is also shown that, in the geometric set-up, the above double cone construction leads to a \textit{unique} autoequivalence $\PPP_P$.
The uniqueness of $\PPP_P$ in the general case is proved in \cite{Anno-Logvinenko-Pn}.
\end{remark}

\begin{proposition}[{c.f.\ \cite[Prop.~2.6]{Huybrechts-Thomas}}]
Let $P$ be a $\IP$-object. Then the associated $\IP$-twist $\PPP_P$ is an autoequivalence.
\end{proposition}

\begin{remark}\label{rem:twist}
In the case of a spherical object, the $\IP^1$-twist associated to it is the square of the spherical twist; see \cite[Prop.~2.9]{Huybrechts-Thomas}.
\end{remark}

\subsection{Formality of single $\IP$-objects}
\label{sec:pn-single}

The following proposition is \autoref{main:pn-single} in the introduction.

\begin{proposition}
\label{prop:pn-single}
Let $P$ be a $\IP^n[k]$-like object in a cocomplete $\kk$-linear dg-enhanced triangulated category with $n,k$ positive integers. Then there are equivalences
\[
 \Genby P \cong \D(\End^*(P)) \quad \text{and}\quad \genby P \cong \D(\End^*(P))^c\,.
\]
\end{proposition}

\begin{proof}
By the definition of a $\IP$-like object, $\End^*(P)=\kk[t]/t^{n+1}$ with $\deg t=k$. Hence, the result follows by \autoref{thm:generatedcat}  together with the following lemma.
\end{proof}

\begin{lemma}\label{lem:Pformal}
For $n, k$ positive integers, the graded algebra $\kk[t]/t^{n+1}$ with $\deg t=k$ is intrinsically formal.
\end{lemma}

\begin{proof}
In order to apply the criterion of \autoref{prop:intrinsically-formal} for intrinsic formality, we have to show the vanishing of the Hochschild cohomology groups $\HH^{q, 2-q}(A,A)$ for $q>2$. 

There is the well-known 2-periodic free resolution 
\[\cdots \to\uA^e \xto{v} \uA^e \xto{u} \uA^e \xto{v} \uA^e \xto{u} \uA^e \xto{m} \uA \to 0\] 
of the underlying non-graded algebra $\uA=\kk[t]/t^{n+1}$ considered as the diagonal bimodule over itself.
Here, $m$ is multiplication in $A$, $u$ is multiplication by $t \otimes 1 - 1 \otimes t$, and $v$ is multiplication by $t^n \otimes 1 + t^{n-1} \otimes t + \cdots + 1 \otimes t^n$; see \cite[Ex.~9.1.4]{Weibel}.
Now one can check easily that this becomes a graded free resolution 
\[
\cdots 
\xto{v} A^e\bigl(-(n+2)k\bigr) \xto{u} A^e\bigl(-(n+1)k\bigr) \xto{v} A^e\bigl(-k\bigr) \xto{u} A^e \xto{m} A \to 0\,.
\]
So we obtain a graded free resolution $F^\bullet$ of the $A^e$-module $A$ where 
\[
F^q=\begin{cases}
A^e\bigl(-i(n+1)k\bigr) & \text{for $q=2i$ even,}\\
A^e\bigl(-(i(n+1)+1)k\bigr) & \text{for $q=2i+1$ odd.}    
    \end{cases}
\]
By \autoref{rem:HHres}, $\HH^{q, 2-q}(A, A)$ is a subquotient of $\Hom_{A^e}^0(F^q, A(2-q))$ so it is sufficient to show that the latter vanishes for $q>2$. 
For $q=2i$ even,
\[
\Hom_{A^e}^0(F^q, A(2-q))=\Hom_{A^e}^0\bigl(A^e, A(2-2i+i(n+1)k)\bigr)=A^{2-2i+i(n+1)k}\,. 
\]
We have $2-2i+i(n+1)k=2+i(nk+k-2)>nk$ for $i\geq2$. But $A$ is concentrated in degrees between $0$ and $nk$, so we get \[\Hom_{A^e}^0(F^{2i}, A(2-2i))=A^{2-2i+i(n+1)k}=0\,.\] The verification that $\Hom_{A^e}^0(F^q, A(2-q))=0$ for $q>2$ odd is similar.  
\end{proof}

\begin{corollary}[{c.f.\ \cite[Prop.~4.2]{Segal}}]
\label{cor:pntwist-spherical-functor}
Let $P$ be a $\IP^n$-object in $\Db(X)$ and $B = \kk[t]$ where $t$ has degree $2$.
Then the functor $F \colon \Db(B) \to \Db(X), B \mapsto P$ is spherical and the spherical twist along $F$ is the $\IP$-twist along $P$.
\end{corollary}

\begin{proof}
This is proved in \cite[Prop.~4.2]{Segal} under the assumption that $\End^\bullet(P)$ is formal. By \autoref{lem:Pformal}, this assumption is always satisfied.

To be precise, Segal's assumption is that the dg-algebra $\End^\bullet(P)$ is formal as a dg-module over $B$, so we have to show that this is implied by its formality as a dg-algebra.  The $B$-module structure is given by choosing an isomorphism $\End^*(P)\cong \kk[s]/s^{n+1}$ and an element $u \in \End^\bullet(P)$ whose cohomology class is mapped to $s$ under this isomorphism. 
Then $t^l$ acts on $\End^\bullet(P)$ by multiplication by $u^l$. Now, we know that $\End^\bullet(P)$ is formal as a dg-algebra. Hence there is a roof
\[
\begin{tikzcd}[column sep=small]
& W^\bullet \ar[dl, "f"'] \ar[dr, "g"] \\
\End^\bullet(P) && \kk[s]/s^{n+1}
\end{tikzcd}
\] where $f$ and $g$ are quasi-isomorphisms of dg-algebras and $f$ is surjective. 
Indeed one can take $f\colon W^\bullet\to \End^\bullet(P)$ to be a cofibrant replacement with respect to the structure of a model category on the category of augmented dg-algebras as described in \cite[\S 4.2]{Keller-functorcats}.

Let $v\in W^\bullet$ be a preimage of $u$ under $f$. 
Then we can equip $W^\bullet$ with the structure of a $B$-algebra by letting $t$ act by $v$ so that $f$ becomes an quasi-isomorphism of $B$-modules. Furthermore, the cohomology class of $v$ is non-zero, hence $g(v)$ is a non-zero multiple of $s$. Therefore, $g$ is a quasi-isomorphism of $B$-modules, too.
\end{proof}

\section{Minimal resolutions of graded algebras}\label{sec:min}

In this section we describe a minimal resolution for certain graded algebras in terms of a tensor presentation, following Eilenberg \cite{Eilenberg} and  Butler and King \cite{Butler-King}. We use this for the computation of the Hochschild cohomology which leads to a sufficient condition for these algebras to be intrinsically formal.

\subsection{Separably augmented algebras and resolutions of diagonal bimodules}

We recall that a $\kk$-algebra $R$ is separable if and only if there is an element $p = \sum x_i \otimes y_i \in R^e$ (called \emph{separability idempotent}), such that $ap=pa$ for all $a \in R$ and $\sum x_iy_i=1$ in $R$.
For general facts on separable algebras see the textbook by Weibel \cite[\S 9.2]{Weibel}.

We denote by $\IN$ the semigroup of non-negative integers. 
A \emph{separably augmented $\kk$-algebra} is an $\IN$-graded $\kk$-algebra such that $R=A^0$ is a separable $\kk$-algebra. 

\begin{remark}\label{rem:Eilenbergalgebra}
Note that $A^+ = \bigoplus_{i>0} A^i$ is the \emph{homogeneous radical} of $A$, that is, the intersection of all homogeneous maximal ideals of $A$. Indeed, every homogeneous maximal ideal of $A$ is of the form $\fm=\fm_0\oplus A^+$. Hence, every separably augmented algebra satisfies the assumptions of \cite[\S 2]{Eilenberg}.
\end{remark}

For a separable $\kk$-algebra $R$ and an $R\hh R$-bimodule $V$,
we denote by $T(V) = R \oplus V \oplus (V \otimes_R V) \oplus \ldots$ its (free) tensor algebra over $R$.
Moreover let $J \subset T(V)$ be the two-sided ideal generated by $V$.
If $V$ carries an $\IN_+$-grading, where $\IN_+$ denotes the positive integers, the tensor algebra inherits a canonical $\IN$-grading given by 
\[
 \deg(v_1\otimes v_2\otimes \dots\otimes v_n)=\deg(v_1)+\deg(v_2)+\dots+\deg(v_n)\,.
\]
Then $T(V)^0=R$, so that $T(V)$ is a separably augmented algebra with $T(V)^+=J$.
 Conversely, every separably augmented algebra with $A^0=R$ has a graded surjection $T(V)\onto A$ for some $R\hh R$-bimodule $V$, for example $V=A^+$. 
Given a separably augmented algebra $A$ and a graded surjection $T(V)\onto A$ with kernel $I$, we call the induced  
isomorphism
$A \cong T(V)/I$ a \emph{tensor presentation} of $A$.
  
Replacing $V$ by $V/(V\cap I)$ if necessary, we may always choose a presentation such that $I \subset J^2$.
In the following, we will also assume that the inclusion $J^n \subset I$ holds for some $n\geq 2$. This is automatic if $A$ is of finite dimension over $\kk$, as it will be in the applications.

\begin{proposition}[{c.f.\ \cite[Prop.~2.4]{Butler-King}}]
\label{prop:minimal-resolution}
Let $A$ be a separably augmented algebra with $A^0=R$.
A minimal resolution $P^\bullet$ of $A$ as a graded $A^e$-module has terms 
$P^m  = A \otimes_R \Tor^A_m(R,R) \otimes_R A$.

Moreover, suppose that $A \cong T(V)/I$ is a tensor presentation with $J^n\subset I\subset J^2$ for some $n\ge 2$.
Then there are isomorphisms of graded $R$-algebras
\[
\Tor^A_{2p}(R,R) = \frac{I^p \cap J I^{p-1} J}{J I^p + I^p J} 
\quad \text{and} \quad
\Tor^A_{2p+1}(R,R) = \frac{J I^p \cap I^p J}{I^{p+1} + J I^p J}.
\]
where the grading on the left-hand side is induced by the grading on $A$ and the one on the right-hand side is induced by the grading on $T(V)$.
\end{proposition}

\begin{proof}
This follows by setting $L=T(V)^e$ in \cite[Prop.~2.4]{Butler-King}. 
Unfortunately, Butler and King assume that the grading of $A$ is induced by the natural grading of $T(V)$ (i.e.\ the elements of $V$ have degree $1$), which will never be the case in our applications.

However, one can check that every step of the proof of \cite[Prop.~2.4]{Butler-King} works in our general graded set-up. 
Indeed, the proof of the equality $P^m  = A \otimes_R \Tor^A_m(R,R) \otimes_R A$ mainly refers to Eilenberg \cite{Eilenberg}, who works in the general graded setting of separably augmented algebras throughout; compare \autoref{rem:Eilenbergalgebra}. 
The arguments in \cite{Butler-King} needed for this equality can all be turned into arguments that also work in our graded set-up using the fact that an object in the category of graded $A$-modules is projective if and only if the underlying non-graded $\underline A$-module is projective; see, for example, \cite[\S 1]{Eilenberg}. 

The computation of $\Tor^A_m(R,R)$ in terms of the ideals $I$ and $J$ is done using the projective resolution 
\[
\cdots \to \frac{JI^n}{JI^{n+1}} \to \frac{I^n}{I^{n+1}} \to \frac{JI^{n-1}}{JI^n} \to \cdots \to \frac{JI}{I} \to A  \to R \to 0\,;
\]
see also \cite{Bongartz} for details. Its differentials are induced by the inclusions of the homogeneous ideals $I$ and $J$, hence are graded homomorphisms.
\end{proof}

\begin{remark}
\label{rem:minimal-resolution-simplification}
Let $P^\bullet \to A$ be the minimal resolution of $A$ as in \autoref{prop:minimal-resolution}.
Note that there is a natural isomorphism 
\[\Hom_{A^e}(P^q, \blank ) = \Hom_{R^e}(\Tor^A_q(R,R),\blank)\,.\] 
\end{remark}

\subsection{Degree criterion for intrinsic formality}

\begin{definition}
For a graded module $M$ we define the \emph{maximal degree} of $M$ as
\[
\maxdeg(M) \coloneqq \max\{\deg(m) \mid \text{non-zero homogeneous } m \in M \}
\]
and analogously the \emph{minimal degree} $\mindeg(M)$.
\end{definition}

\begin{proposition}
\label{prop:max-min-degree}
Let $A$ be a separably augmented algebra and let $q\in \IN$.
If 
\[
\maxdeg(A)+q-2 < \mindeg(\Tor_q^A(R,R))
\]
then the Hochschild cohomology $\HH^{q,2-q}(A,A)$ vanishes.
In particular, if this inequality holds for all $q>2$, then $A$ is intrinsically formal.
\end{proposition}

\begin{proof}
By \autoref{rem:HHres} and \ref{rem:minimal-resolution-simplification}, 
the Hochschild cohomology $\HH^{q,2-q}(A,A)$ is a subquotient of
\[
\Hom_{R^e}( \Tor_q^A(R,R), A\ds{2-q})\,. 
\]
Hence, it is sufficient to show the vanishing of this $\Hom$-space.
But there cannot be any non-zero homomorphism of degree zero, since the minimal degree of the source is smaller than the maximal degree of the target. 

Recall that $\HH^{q,2-q}(A,A)=0$ for $q>2$ implies intrinsic formality of $A$ by \autoref{prop:intrinsically-formal}. 
\end{proof}

We will use \autoref{prop:max-min-degree} to prove intrinsic formality of a given separably augmented algebra using a suitable tensor representation, namely in the case of endomorphism algebras of configurations of $\IP$-objects.  

\begin{remark}
\label{rem:mindeg}
Let $A$ be an $\IN$-graded $\kk$-algebra $P$ a graded $A$-module, and $M,N \subset P$ graded submodules.
Then the following rules hold:
\begin{itemize}
\item $\mindeg(M+N) = \min \{ \mindeg(M),\mindeg(N) \}$; hence $\mindeg(M) = \min\{ \deg(m_i) \}$ for $M = {}_A\genby{m_1,\ldots,m_l}$ with $m_i$ homogeneous;
\item $\mindeg(M\cap N) \geq \max\{ \mindeg(M),\mindeg(N) \}$;
\item $\mindeg(P/M) \geq \mindeg(P)$, with equality if $\mindeg(M) > \mindeg(P)$.
\end{itemize}
If $I,J \subset A$ are ideals then we have additionally:
\begin{itemize}
\item $\mindeg(I\cdot J) \ge \mindeg(I) + \mindeg(J)$.
\end{itemize}
Let $A$ be a separably augmented algebra with tensor representation $A\cong T(V)/I$ and let $J=A^+$ so that $\mindeg(J)=\mindeg(V)$. 
By \autoref{prop:minimal-resolution} and the above rules, we get 
\[
\begin{split}
\mindeg\Tor^A_{2p}(R,R) & \geq \mindeg(I^p \cap J I^{p-1} J) \\ & \geq
                           \max\{p\mindeg(I), 2\mindeg(J)+(p-1)\mindeg(I)\}\,,\\
\mindeg\Tor^A_{2p+1}(R,R) & \geq \mindeg(J I^p \cap I^p J) \geq p\mindeg(I)+\mindeg(J)\,.
\end{split}
\]
\end{remark}

\section{Configurations of $\IP$-objects}
\label{sec:pobjects}

\begin{definition}
\label{def:tree}
Let $\cT$ a triangulated category and let $Q$ be a graph. Our convention for a graph is that we allow at most one edge joining two given vertices $i\neq j$ and no edge from a vertex to itself. 
A \emph{$Q$-configuration} of objects in $\cT$ is  
a collection of indecomposable objects $P_i$, one for every vertex $i$ of $Q$, such that, for all $i\neq j$, we have $\Hom^*(P_i,P_j)= \Hom^*(P_j,P_i)=0$ if $i$ and $j$ are not adjacent and 
\[
 \dim_{\kk} \Hom^*(P_i,P_j)= \dim_{\kk} \Hom^*(P_j,P_i)=1 
\]
if $i$ and $j$ are connected by an edge.
A \emph{tree} is a graph in the sense above without loops. 
Given a $Q$-configuration $\{P_i\}$ for a tree $Q$ we also say that the objects $P_i$ \emph{form the tree $Q$}.
\end{definition}

\begin{remark}
\label{rem:tensorpres}
Let $n,k$ be positive integers and let $P_1,\ldots,P_m$ be a $Q$-configuration of $\IP^n[k]$-like objects such that $\Hom^{\leq 0}(P_i,P_j)$ is zero for $i\neq j$.
Then $A = \bigoplus_j \End^j(\bigoplus_i P_i)$ is a separably augmented algebra,
since $R = A^0 = {}_{\kk}\genby{e_1,\ldots,e_m}$ is spanned by the mutually orthogonal idempotents $e_i \coloneqq \id_{P_i}$.

For each $i$, denote $t_i$ a non-zero map in $\Ext^k(P_i,P_i)$, which is unique up to multiplication with a unit. 
By assumption, for any two $P_i$ and $P_j$ adjacent in $Q$, there is a unique positive degree $h_{ij}$ such that $\Ext^{h_{ij}}(P_i,P_j) = \kk \cdot a_{ij}$.
Let $V \subset \End^*(\bigoplus_{i=1}^m P_i)$ be the graded $\kk$-subvector space spanned by all $t_i$ and $a_{ij}$.
This gives a graded surjection $T(V) \onto A$, hence a tensor presentation $A = T(V)/I$ by some homogeneous ideal $I$.
\end{remark}

\subsection{Formality of configurations of $\IP$-objects}

\begin{proposition}
\label{prop:pn-moregeneral}
Let $Q$ be a graph and let 
 $P_1,\ldots,P_m$ be a $Q$-configuration consisting of $\IP^n[k]$-like objects in a $\kk$-linear triangulated category with $n,k$ integers with $n,k\geq2$. 
Assume that there exists an integer $h$ with $\frac{nk}2\le h\le nk$ and $\gcd(k,h)>1$ such that 
\[
\Hom^*(P_i,P_j) = \kk[-h] \quad \text{for all adjacent $P_i$ and $P_j$.}
\]
Then $A = \End^*(\bigoplus_{i=1}^m P_i)$ is intrinsically formal.
\end{proposition}

\begin{proof}
We will use the tensor presentation of $A = T(V)/I$ as in \autoref{rem:tensorpres} together with the notation $e_i = \id_{P_i}$, $\End^k(P_i) = \kk \cdot t_i$ and $\Ext^h(P_i,P_j) = \kk \cdot a_{ij}$ for adjacent $P_i$ and $P_j$. Here $V$ is the graded vector space spanned by all $t_i$ and $a_{ij}$. Recall that $J = T(V)^+$. 

Note that the homogeneous elements $1, t_i, \ldots, t_i^n, a_{ij}$ constitute a basis of $A$ as a $\kk$-vector space. Hence, 
\begin{equation}
\tag{$A$}
\label{eq:maxdegA}
\maxdeg(A) = nk \le 2h\,.
\end{equation}

The only elements of $I$ not involving an $a_{ij}$ lie in the ideal generated by the elements $t_i^{n+1}$. 
Indeed $t_i^l\neq 0$ for $l<n+1$ by the definition of a $\IP^n[k]$-object. Furthermore, for $i\neq j$, the tensor product $t_i\otimes t_j$ 
already vanishes as an element of $V\otimes_R V\subset T(V)$ as $t_it_j = t_i\otimes t_j=(t_i e_i)\otimes t_j=t_i\otimes (e_i t_j)=0$. 
Hence, the minimal degrees of $I$ and $J$ are
\begin{align*}
\mindeg(I) \geq &\ \min\{ \deg(t_i^{n+1}), \deg(a_{ij}a_{jl}), \deg(a_{ij}t_j) \}  = h+k\,, \\
\mindeg(J) = &\   \min\{ \deg(t_i), \deg(a_{ij}) \} = k\,,
\end{align*}
where the assumption that $n \geq 2$, hence $h\ge k$, is used.
Hence, by \autoref{rem:mindeg}, we get
\begin{align}
\tag{e}
\label{eq:even}
\mindeg \Tor^A_{2p}(R,R) \geq &\  \max\{ p(h\!+\!k),  2k + (p\!-\!1)(h\!+\!k) \}=p(h+k),\\
\tag{o}
\label{eq:odd}
\mindeg \Tor^A_{2p+1}(R,R) \geq &\  p(h+k)+k\,.
\end{align}

We can now confirm that the assumptions of \autoref{prop:max-min-degree} are satisfied for $q\ge 4$. For $q=2p$ with $p\ge 2$, we have
\[
\maxdeg(A)+q-2 \overset{\eqref{eq:maxdegA}}{\le} 2h+2p-2\overset{(2\le p)}< ph+2p \overset{(2\le k)}\le ph+pk \overset{\eqref{eq:even}} \leq \mindeg\Tor_{2p}^A(R,R)\,.
\]
Similarly, for $q=2p+1$ with $p\ge 2$,
\[
\maxdeg(A)+q-2 \overset{\eqref{eq:maxdegA}} = 2h+2p-1\overset{(2\le p)} <
 ph+2p \underset{(2\le k)} < ph+pk + k \underset{\eqref{eq:odd}} \leq \mindeg\Tor_{2p+1}^A(R,R)\,.
\]

Hence, $\HH^{q,2-q}(A,A)$ vanishes for $q\ge 4$. In order to apply Kadeishvili's criterion for intrinsic formality (see \autoref{prop:intrinsically-formal}), all that is left is to show that $\HH^{3,-1}(A,A)=0$. 
This can be done using the Bar resolution $B^\bullet$; see \autoref{sec:hochschild}. Indeed, $A$ is concentrated in degrees divisible by $\gcd(k,\frac{nk}2)>1$.
Hence, the same holds for $B^3=A^{\otimes 5}$. 
Thus, there is no non-trivial degree-zero homomorphism $A^{\otimes 5}\to A(-1)$.   
\end{proof}

The previous proposition does not cover the interesting case of configurations of spherical objects, which we treat in the following proposition.

\begin{proposition}
\label{prop:spherical-moregeneral}
Let $Q$ be a graph and let 
 $P_1,\ldots,P_m$ be a $Q$-configuration consisting of $k$-spherelike objects in a $\kk$-linear triangulated category with $k\geq4$.
Moreover, assume that, for adjacent $P_i$ and $P_j$,
\[
\Hom^*(P_i,P_j) = \kk[-h_{ij}] \quad\text{with } \abrunden{\frac k2}\le h_{ij}\le k \,.
\]
Then $A = \End^*(\bigoplus_{i=1}^m P_i)$ is intrinsically formal.
\end{proposition}

\begin{proof}
This can be shown along the same lines as the proof of \autoref{prop:pn-moregeneral}, but some of the estimates change.
The minimal degrees of the ideals $I$ and $J$ become
\begin{align*}
\mindeg(I) & \geq \min\{ \deg(t_i^{2}), \deg(a_{ij}a_{jl}), \deg(a_{ij}t_l) \}  \ge 2 h \,,\\ 
\mindeg(J) & =  \min\{ \deg(t_i), \deg(a_{ij}) \} \ge h\,,
\end{align*}
where we abbreviate $h \coloneqq \abrunden{\frac k2}$. 
Note that $h\geq2$ by the assumption on $k$.
Hence, the minimal degrees in the minimal projective resolution are now
\begin{align}
\tag{e}
\label{eq:even3}
\mindeg \Tor^A_{2p}(R,R) \geq &\  \max\{ 2ph, 2h + 2(p-1)h \}=2ph\,,\\
\tag{o}
\label{eq:odd3}
\mindeg \Tor^A_{2p+1}(R,R) \geq &\ 2ph+h\,.
\end{align}
Furthermore, note that 
\begin{equation}
\tag{$A$}
\label{eq:maxdeg3}
\maxdeg(A) = k \leq 2h+1\,.
\end{equation}

We will check that the assumptions of \autoref{prop:max-min-degree} are satisfied for $q > 2$, which concludes the proof. Indeed, for $q=2p$ with $p\ge 2$,
\begin{align*}
\maxdeg(A)+q-2\overset{\eqref{eq:maxdeg3}} \leq  2h+2p-1 < 2h+2p \leq  2ph \overset{\eqref{eq:even3}} \leq \mindeg\Tor_{2p}^A(R,R)\,.
\end{align*}
To see this, we still need $h+p\leq ph$. This inequality is equivalent to $\frac{h}{h-1} \leq p$, as $h\geq 2$, and holds as $p\geq 2$.
Similarly, for $q=2p+1$ with $p\ge 1$, we have 
\begin{align*}
\maxdeg(A)+q-2\overset{\eqref{eq:maxdeg3}} \leq 2h+2p <  2ph+h \overset{\eqref{eq:even}} \leq \mindeg\Tor_{2p+1}^A(R,R)\,.
\end{align*}
Here, the middle $2h+2p <  2ph+h$ is equivalent to $\frac{h}{h-1} < 2p$, as $h\geq2$, hence the inequality holds due to $p \geq1$. 
\end{proof}

\begin{remark}
\label{rem:CYproperty}
If we assume in \autoref{prop:pn-moregeneral} that the $P_i$ are $\IP^n[k]$-objects (not just $\IP$-like), the assumption $\frac{nk}2\le h\le nk$ already implies $h=\frac{nk}2$; compare the proof of \autoref{cor:tree} below. 
In this case, the assumption $\gcd(k,h)>1$ is automatically fulfilled if $n$ is even or $k$ is a multiple of $4$.

Similarly, if we assume the objects in \autoref{prop:spherical-moregeneral} to be spherical, the assumption $\abrunden{\frac k2}\le h_{ij}\le k$ already implies $h\in {\textstyle \left\{\abrunden{\frac k2}, \aufrunden{\frac k2} \right\}}$.
\end{remark}

\begin{remark}
For $k=1$, the assertion of \autoref{prop:spherical-moregeneral} has to be false.
To see a counterexample, consider an elliptic curve $E$. 
Then $\Db(\Coh(E))$ is generated by the $A_2$-sequence of $1$-spherical sheaves $\reg_E$ and $\reg_p$ for any $p\in E$. 
Indeed, out of these two sheaves one can construct all the line bundles $\reg(n\cdot p)$ by successive cones, 
and the line bundles $\reg(n\cdot p)$ contain an ample sequence. 
Now, one can check that the graded endomorphism algebra $\End^*(\reg_E\oplus \reg_p)$ is the same, regardless of the chosen elliptic curve $E$ and point $p\in E$. 
However, two non-isomorphic elliptic curves $E\not\cong E'$ always have non-equivalent bounded derived categories; see, for example, \cite[Cor.~5.46]{Huybrechts}. 

For $k=2$ and $3$, we still expect intrinsic formality, as in the case of algebras coming from $A_m$-configurations of such objects; see \cite{Seidel-Thomas}.
\end{remark}

The following corollary in combination with \autoref{thm:generatedcat} gives \autoref{main:pn-many} in the introduction.

\begin{corollary}\label{cor:tree}
Let $\{P_i\}$ be a tree of $\IP^n[k]$-objects in a cocomplete  dg-enhanced triangulated category $\cT$ with 
\begin{itemize}
\item either $n,k\ge 2$, $nk$ even, and $\gcd(k,{\textstyle \frac{nk}2})>1$;
\item or $n=1$ and $k\ge 4$.
\end{itemize}
 Then the thick subcategory $\genby{\{P_i\}}$ is independent of the ambient category $\cT$.
\end{corollary}

\begin{proof}
Replacing the objects $P_i$ by appropriate shifts $P_i[n_i]$, we may assume that they satisfy the assumptions on $\Hom^*(P_i, P_j)$ of \autoref{prop:pn-moregeneral} and \autoref{prop:spherical-moregeneral}, respectively.
For the restrictions on $n$ and $k$ see \autoref{rem:CYproperty}.
Denote by $Q$ the underlying tree. We may start with some edge $i\in Q$ and set $n_{i}=0$. 
By definition of a $Q$-configuration of objects, for adjacent $i$ and $j$, we have
\[
 \Hom^*(P_i,P_j)=\kk[-a]\,,\ \Hom^*(P_j, P_i)=\kk[-b]
\]
for some $a,b\in \IZ$.
Note that we assume the objects to be $\IP$-objects (not just $\IP$-like). Hence, Serre duality gives 
$a+b=nk$. Hence, in the case where $nk$ is even, we may set $n_j=a-h$, so after replacing $P_j$ by $P_j[a-h]$ we get 
\[
 \Hom^*(P_i,P_j)=\kk[-h]= \Hom^*(P_j, P_i)\,.
\]
Since, by assumption, $Q$ has no loops, there is no obstruction to extending this procedure to the whole of $Q$. The case where $n=1$ and $k\ge 5$ is odd works similarly; compare \cite[\S 4c]{Seidel-Thomas}.

Now we can use \autoref{prop:pn-moregeneral} and \ref{prop:spherical-moregeneral} together with \autoref{cor:intformal} to conclude that $\genby{\{P_i\}}\cong \D(A)^c$ where $A=\End(\bigoplus_i P_i)$.  
\end{proof}

\begin{remark}
There are results, analogous to those of this subsection, on the formality of the endomorphism algebras of configurations of $\IP^n[k]$-objects for $k$ negative. 
To see this, one can use that a non-positively graded algebra is intrinsically formal as soon as 
\[
\mindeg(A)+q-2 > \maxdeg(\Tor_q^A(R,R))\quad \text{for $q\ge 3$}\,,
\]
which is analogous to \autoref{prop:max-min-degree}.

We chose to concentrate on $\IP^n[k]$-objects with $k$ positive, since those with negative $k$ are rare in practice, for example, negative Calabi--Yau objects cannot appear in derived categories of smooth varieties; compare \cite[Lem.\ 1.7]{Hochenegger-Kalck-Ploog-RT}.
\end{remark}

\subsection{Actions induced by $A_m$-configurations of $\IP$-objects}
\label{subsec:faithful}

Seidel and Thomas \cite{Seidel-Thomas} used the formality of $\End(\bigoplus_i P_i)$, where $P_1,\dots, P_m$ is an $A_m$-configuration of $k$-spherical objects in a cocomplete dg-enhanced triangulated category $\cT$, in order to prove that the induced action of the braid group $B_{m+1}$ on $\cT$ is faithful. This means that the subgroup $\genby{\TTT_{P_1},\dots,\TTT_{P_m}} \subset \Aut(\cT)$ generated by the spherical twists is isomorphic to $B_{m+1}$. 
Consider the $P_i$ as $\IP^1[k]$-objects and the associated $\IP$-twists are then the square of the spherical twists: $\PPP_{P_i}\cong \TTT_{P_i}^{\,2}$; compare \autoref{rem:twist}.   
It follows from the description of the group spanned by the squares of the standard generators of the braid group \cite{Collins--squares} that the only relations between the $\IP$-twists are
the commutativity relations 
\[
 \PT_i\PT_j=\PT_j\PT_i\quad \text{for $|i-j|>1$.}
\]
Hence, it makes sense to conjecture the following more general faithfulness result.

\begin{conjecture}
Let $P_1,\dots, P_m$ be an $A_m$-configuration of $\IP^n[k]$-objects with $k\ge 2$.
Then the only relations between the associated $\IP$-twists $\PT_i:=\PT_{E_i}\in \Aut(\cT)$ are the commutativity relations 
\[
 \PT_i\PT_j=\PT_j\PT_i\quad \text{for $|i-j|>1$.}
\]
\end{conjecture}

It is easy to see that, for two $\IP^n[k]$-objects with vanishing graded Hom-space between them, the associated $\IP$-twists commute; see \cite[Cor.\ 2.5]{Krug-autos}. 
Hence, the unknown and probably difficult part of the conjecture is that there are no further relations between the twists associated to an $A_m$-configuration of $\IP$-objects.

By \autoref{cor:tree}, it would be sufficient to consider one particular example of an $A_m$-con\-fig\-ura\-tion of $\IP^n[k]$-objects in order to prove (or disprove) the conjecture for a fixed value of $m$, $n$ and $k$ with $nk$ even and $\gcd(k,\frac {nk}2)>1$.

\section{The triangulated category generated by the structure sheaf}
\label{sec:OX}

Let $X$ be a smooth projective variety over $\IC$.
Note that the graded endomorphism algebra $\End^*(\reg_X)$ coincides with the cohomology algebra $\Ho^*(\reg_X)$ where the multiplication is given by the cup product. 
This algebra, sometimes called the  \emph{homological unit} of $X$ is conjectured to be a derived invariant of the variety $X$; see \cite{Abuaf-homologicalunit}.
In this section we will show that the generated thick triangulated category $\langle \reg_X\rangle\subset \Db(\Coh(X))$ only depends on this graded algebra. 

Actually, we show this statement for any compact complex manifold $X$ which \emph{satisfies the $\partial\bar\partial$-lemma}, that is, $X$ has the following property.

\emph{Let $\omega$ be a complex-valued differential form on $X$ which is $\partial$-closed and $\bar\partial$-closed. If $\omega$ is $\partial$-exact or $\bar\partial$-exact, then it is already $\partial\bar\partial$-exact, which means that there is a differential form $\chi$ with $\partial\bar\partial \chi=\omega$.}

Every compact K\"ahler manifold satisfies the $\partial\bar\partial$-Lemma; see, for example, \cite[\S 6.1]{Voisin-book1}. 
However, there are compact complex manifolds satisfying the $\partial\bar\partial$-Lemma which are not K\"ahler. 
Still they share many properties of compact K\"ahler manifolds; to get an impression see \cite{Deligneetal}, \cite{Angella}, \cite{Anthesetal}.

\begin{theorem}\label{thm:unit}
Let $Y$ be a compact complex manifold satisfying the $\partial\bar\partial$-Lemma. Let $\Db_{\textrm{coh}}(Y)$ be the subcategory of complexes with bounded and coherent cohomology in $\D(\Mod(Y))$ and $\genby{\OO_Y}\subset \Db_{\textrm{coh}}(Y)$ the thick subcategory generated by $\reg_Y$.
Then there is an equivalence $\genby{\OO_Y}\cong \D(\Ho^*(\reg_Y))^c$.
\end{theorem}

In this case, by the GAGA principle, we get an equivalence $\Db(\Coh(X))\cong \Db_{\textrm{coh}}(\Xan)$; see \cite[Thm.\ 2.2.10]{Caldararu-thesis}.  
Consequently, we also obtain $\genby{\OO_X}\cong \D(\Ho^*(\reg_X))^c$. 

We denote by $A^{0,\bullet}$ the \emph{Dolbeault complex} on the compact complex manifold $Y$. Its terms $A^{0,p}$ are the anti-holomorphic $p$-forms on $Y$ and its differential is given by $\bar \partial$.

\begin{proposition}[{\cite[Thm.\ 8]{Neis-Taylor}}]\label{lem:Dolbeault}
Let $Y$ be a connected complex manifold satisfying the $\partial\bar\partial$-Lemma.
Then the Dolbeault complex $A^{0,\bullet}$ on $Y$ is formal. 
\end{proposition}

\begin{proof}[Proof of \autoref{thm:unit}]
Note that $\OO_Y$ is a compact object in $\Db_{\textrm{coh}}(Y)$. 
Thus, in view of \autoref{lem:Dolbeault} and \autoref{thm:generatedcat}\eqref{thm:i}, it is enough to find a pretriangulated dg-category $\EE$ together with an exact equivalence $\alpha\colon \Ho^0(\EE)\to \Db_{\textrm{coh}}(Y)$ and an object $E\in \EE$ with $\alpha(E)\cong \OO_Y$ and $\Hom_{\EE}^\bullet (E,E)\cong A^{0,\bullet}$.

A dg-enhancement of $\Db_{\textrm{coh}}(Y)$ is given by the category $\EE=\cP_A$ of \cite{Block-dg}; see, in particular, \cite[Thm.\ 4.3]{Block-dg}. 
The objects of $\cP_A$ are given by pairs $(M, \nabla)$ consisting of a graded module $M$ over $\cA=A^{0,0}$ and a connection $\nabla\colon M\otimes_\cA A^{0,\bullet}$ satisfying some additional conditions; see \cite[Def.\ 2.4]{Block-dg} for details.  
We consider the object $E=(\cA, \bar\partial)\in \cP_A$. Indeed, $\alpha(E)=\cA_Y^{0,\bullet}$ where $\alpha\colon \Ho^0(\cP_A)\isom\Db_{\textrm{coh}}(Y)$ is the equivalence constructed in \cite[Lem.\ 4.5]{Block-dg} and  $\cA_Y^{0,\bullet}$ denotes the Dolbeault complex of \emph{sheaves} (not their global sections as in $A^{0,\bullet}$). The complex $\cA_Y^{0,\bullet}$ is a resolution of $\OO_Y$; see, for example, \cite[Prop.\ 4.19]{Voisin-book1}. Hence, $\alpha(\EE)\cong \OO_Y$. 
The fact that $\Hom_{\EE}^\bullet(E,E)\cong A^{0,\bullet}$ follows directly from the definition of the Hom-complexes in the category $\cP_A$; see \cite[Def.\ 2.4]{Block-dg}. 
\end{proof}

\section{Examples of configurations of $\IP$-objects}\label{sec:example}

In the following examples, we assume that the characteristic of the field $\kk$ does not divide the order $n!$ of the group $\sym_n$.

\subsection{Trees of $\IP$-like objects on symmetric quotient stacks}

We recall a construction of $\IP^n[k]$-like objects from $k$-spherelike objects, which is essentially due to Ploog and Sosna in \cite{Ploog-Sosna}.
Let $X$ be a smooth projective variety and $E$ be a $k$-spherelike object in $\Db(X)$.
Consider the $n$-fold cartesian product $X^n$ with its projections $\pi_i \colon X^n \to X$.
Then we define
\[
E^{\boxtimes n} = \pi_1^* E \otimes \cdots \otimes \pi_n^* E \in \Db(X^n).
\]
There is a natural action on $X^n$ by the permutation group $\sym_n$. 
Actually, we can turn $E^{\boxtimes n}$ into an object $E\poslinn$ in the equivariant derived category $\Db_{\sym_n}(X^n)$ by equipping $E^{\boxtimes n}$ with the canonical linearisation given by permutation of the tensor factors.
By $E\neglinn\in\Db_{\sym_n}(X^n)$, we denote the object $E^{\boxtimes n}$ equipped with the linearisation which differs from the canonical one by the non-trivial character (also known as \emph{sign} or \emph{alternating representation}) $\alt$ of $\sym_n$.
Let $[X^n/\sym_n]$ be the quotient stack of $X^n$ by the permutation action of $\sym_n$.
By the definition of sheaves on a quotient stack, there is an equivalence $\Db_{\sym_n}(X^n) \cong \Db([X^n/\sym_n])$.

\begin{remark}
The above construction also works for $E$ inside $\Db(A) = \Db(\mod(A))$ where $A$ is a $\kk$-algebra. Instead of the cartesian product $X^n$ we consider $A^{\otimes n}$ and instead of the equivariant derived category we consider the derived category $\Db(\sym_n \# A^{\otimes n})$ of the algebra $\sym_n \# A^{\otimes n}$ which is known as the \emph{skew group algebra} in the literature.

More generally, there is the concept of the symmetric power $S^n\cT$ of a (dg-enhanced) triangulated category $\cT$ due to  Ganter and Kapranov \cite{Ganter-Kapranov}. This covers both of the constructions above, since $S^n\Db(X)\cong\Db_{\sym_n}(X^n)$ and $S^n\Db(A)\cong\Db(\sym_n\# A)$ for a variety $X$ and an algebra $A$, respectively.
\end{remark}

\begin{remark}
\label{rem:geometric-interpretation}
If $X$ is a surface, then $\Db_{\sym_n}(X^n)$ has a very geometric interpretation. 
Namely, the derived McKay correspondence of  Bridgeland,  King and Reid \cite{Bridgeland-King-Reid} and Haiman \cite{Haiman} gives an equivalence $\Db_{\sym_n}(X^n)\cong \Db(X^{[n]})$ where $X^{[n]}$ denotes the Hilbert scheme of $n$ points on $X$; see also \cite[\S 1.3]{Sca}.
In \autoref{app:geometric-pns} we will describe the corresponding $\IP$-objects in $\Db(X^{[n]})$.
\end{remark}

\begin{proposition}[{c.f.\ \cite[Lem.~4.2]{Ploog-Sosna}}]
\label{Prop:gradedhom}
Let $n,k\in \IN$ with $n\ge 2$ and $k$ even.
Let $X$ be a smooth projective variety and let $E$ be a $k$-spherelike object in $\Db(X)$.
Then $E\poslinn, E\neglinn \in \Db_{\sym_n}(X^n)$ are $\IP^n[k]$-like objects.

Moreover, if $E$ is a $k$-spherical object 
(e.g.\ if $E$ is spherelike and $X$ is a  Calabi--Yau variety of dimension $k$), 
then $E\poslinn$ and $E\neglinn$ are $\IP^n[k]$-objects.

Finally, let $E, F\in \Db(X)$ be objects with $\Hom^*(E,F)=\kk[-m]$ for some $m\in \IN$. Then
\begin{align*}
\Hom^*(E\poslinn, F\poslinn)= \Hom^*(E\neglinn, F\neglinn)=\begin{cases}\kk[-nm] \quad &\text{if $m$ is even,} \\                          0  \quad &\text{if $m$ is odd.}
                          \end{cases}
\\
\Hom^*(E\neglinn, F\poslinn)= \Hom^*(E\poslinn, F\neglinn)=\begin{cases}0\quad &\text{if $m$ is even,}\\                           \kk[-mn] \quad &\text{if $m$ is odd.}
                          \end{cases}
\end{align*}
\end{proposition}

\begin{proof}
 All of this follows from the equivariant K\"unneth formula which says that
\begin{align*}
\Hom^*(E\poslinn, F\poslinn)= \Hom^*(E\neglinn, F\neglinn)&= S^n\Hom^*(E,F)\,.
\\
\Hom^*(E\neglinn, F\poslinn)= \Hom^*(E\poslinn, F\neglinn)&= \medwedge{n}\Hom^*(E,F)\,.
\end{align*}
Note that both the symmetric and the exterior product are formed in the graded sense. For example, if $m\in \IN$ is odd, then
\[
 S^n(\kk[-m])\cong \bigl(\medwedge{n} \kk\bigr)[-mn]=0\,.\qedhere
\]
\end{proof}

\begin{corollary}\label{Cor:inducedtree}
Let $\{E_i\}$ be $k$-spherelike objects in $\Db(X)$ which form a tree $Q$. Then there is a choice of signs $\eps_i=\pm 1$ such that the $\IP^n[k]$-like objects $E_1^{\eps_1\linearised{n}},\ldots,E_m^{\eps_m\linearised{n}}$ form the same tree $Q$.
Additionally, if the $E_i$ are $k$-spherical, then this is a configuration of $\IP^n[k]$-objects.
\end{corollary}

\begin{proof}
We start with some vertex $i_0$ of $Q$ and set $F_{i_0} \coloneqq E_{i_0}\poslinn$. 
Then, given an adjacent $j$ of $i_0$, the graded Hom-space $\Hom^*(E_{i_0}, E_j)$ is one-dimensional, 
hence concentrated in one degree, say $d$. 
We set $\eps_j \coloneqq (-1)^{d}$ and $F_j \coloneqq E_j^{\eps_{j}\linearised{n}}$. 
By \autoref{Prop:gradedhom}, $F_{i_0}$ and $F_j$ are $\IP^n[k]$-like objects and $\Hom^*(F_{i_0}, F_j)=\kk[dn]$ is one-dimensional. 
In other words, $F_{i_0}$ and $F_j$ form an $A_1$-tree of $\IP^n[k]$-like objects. 
Since $Q$ is a tree, we can continue inductively and end up with a $Q$-tree $\{F_j=E_j^{\eps_j\linearised{n}}\}$ of $\IP^n[k]$-like objects.      
\end{proof}

\begin{remark}
\label{rem:closing-condition}
The corollary yields also more general configurations of $\IP$-like objects, provided that for each cycle the signs can be attributed consistently. We do not spell out the details, but provide an example.

By a result of Kodaira, cycles of $(-2)$-curves $C_i$ appear as singular fibres in elliptic fibrations of surfaces; see the book \cite[\S V.7]{Barth-etal} by Barth, Hulek, Peters and van den Ven. 
Hence, such a cycle forms a cycle of spherical objects $\OO_{C_i}$, as $\Hom^*(\OO_{C_i},\OO_{C_j})$ is non-zero if and only if the curves $C_i$ and $C_j$ intersect.
Consequently, these objects induce a cycle of $\IP^n$-objects, provided the cycle is of even length.
\end{remark}

\subsection{A geometric example of a tree of $\IP$-objects}
\label{app:geometric-pns}

Let $Q$ be a tree and $X$ be a smooth quasi-projective surface together with a $Q$-configuration of $(-2)$-curves. This means that, for every vertex $i$ of the tree $Q$, there is a $(-2)$-curve $\IP^1\cong C_i\subset X$, $C_i$ and $C_j$ intersect in one point if there is an edge joining $i$ and $j$ and they do not intersect otherwise.
Note that such a configuration might not exist for any given tree $Q$; see \cite[\S 6]{Hochenegger-Ploog} for sufficient criteria.

The objects $\OO_{C_i} \in \Db(X)$ form a $Q$-tree of $2$-spherical objects with 
$\Hom^*(\OO_{C_i}, \OO_{C_j})=\kk[-1]$ for adjacent $i$ and $j$; see \cite[Ex.\ 3.5]{Seidel-Thomas}. 
By \autoref{Cor:inducedtree}, there is an induced $Q$-tree of $\IP^n$-objects of the form $\OO_{C_i}^{\eps_i\linearised n}$ in $\D^b_{\sym_n}(X^n)$. Here we have to choose opposite signs $\eps_i =-\eps_j$ for adjacent $i$ and $j$, as the graded Hom-space is concentrated in the odd degree 1.

We use the derived McKay correspondence as mentioned in \autoref{rem:geometric-interpretation}
(recall that we omit $R$ and $L$ in front of derived functors)
\[
 \Phi \coloneqq p_*\circ q^*\colon \Db(X^{[n]})\isom \Db_{\sym_n}(X^n)
\]
to interpret this as a tree of $\IP^n$-objects on the Hilbert scheme $X^{[n]}$. 
Here we denote by $q\colon \cZ\to X^{[n]}$  and $p\colon \cZ\to X^n$ the projections from the universal family of $\sym_n$-clusters $\cZ\subset X^{[n]}\times X^n$. 
Hence there is a commutative diagram 
\[
\begin{tikzcd}
\cZ \ar[r, "p"'] \ar[d, "q"'] & X^n \ar[d, "\pi"] \\
X^{[n]} \ar[r, "\mu"] & X^{(n)}
\end{tikzcd}
\]
where $X^{(n)} \coloneqq X^n/\sym_n$ is the symmetric product, $\pi$ is the $\sym_n$-quotient morphism, and $\mu$ is the Hilbert--Chow morphism. Furthermore, $\cZ\cong (X^{[n]}\times_{ X^{(n)}}X^n)_{\mathsf{red}}$ is the reduced fibre product of this diagram.
Note that every closed subscheme $C$ of $X$
induces a canonical closed embedding $C^{[n]}\hookrightarrow X^{[n]}$.

\begin{proposition}
\label{prop:CMcKay}
For $C\subset X$ a smooth curve, we have $\Phi(\OO_{C^{[n]}})\cong \OO_C\poslinn$. 
\end{proposition}

\begin{proof}
For a smooth curve $C$, the Hilbert--Chow morphism $C^{[n]}\to C^{(n)}$ is an isomorphism. 
So, $C^{[n]} \isom C^{(n)} \into X^{(n)}$ is a closed embedding with image $C^{(n)}$. 
Consequently, using the diagram above, $p\colon \cZ\to X^n$ maps $q^{-1}C^{[n]}$ isomorphically to $C^n$. 
Hence, $\Phi(\OO_{C^{[n]}})\cong \OO_{C^n}\cong\OO_C^{\linearised n}$. 
\end{proof}

Let $i$ and $j$ be two adjacent vertices of $Q$. 
Then $C_i\cap C_j$ is a reduced point, hence cannot contain a subscheme of length $n\ge 2$. 
Thus, $C_i^{[n]}$ and $C_j^{[n]}$ do not intersect inside $X^{[n]}$. 
So we see geometrically that $\Hom^*(\OO_{C_i}\poslinn, \OO_{C_j}\poslinn)=0$ for support reasons.

For $n=2$, we give a concrete description of the image of $\OO_{C}\neglinn$ under the McKay correspondence,
where $C=C_i$ is one of the rational curves.
Denote by $\delta\colon X\to X^{(2)}$ the diagonal embedding into the symmetric product.
Then $E \coloneqq \mu^{-1}\delta(X)$ is the exceptional divisor of $X^{[2]}\to X^{(2)}$.
There is a line bundle $\LL\in \Pic(X^{[2]})$ such that $\LL^2\cong \OO(E)$; see \cite[Lem.\ 3.7]{Lehn-chern}.

We summarise this situation in the following diagram, consisting of a blow-up square on the right and its restriction to $C$, so both are cartesian:
\begin{equation}
\tag{$\square$}
\label{diag:E}
\begin{gathered}
\begin{tikzcd}
\llap{$\mu^{-1}\delta(C) \eqqcolon\,$} \Sigma_C \ar[r, hook]  \ar[d] & E \ar[r, hook, "\iota"'] \ar[d, "\rho"'] & X^{[2]} \rlap{$\,= \Bl_{\delta(X)}(X^{(2)})$} \ar[d, "\mu"] \\
C \ar[r, hook] & X \ar[r, hook, "\delta"] & X^{(2)}
\end{tikzcd}
\end{gathered}
\end{equation}
Note that the restrictions $\rho\colon E\to X$ and $\Sigma_C \to C$ of $\mu$ are $\IP^1$-bundles
(actually, $\Sigma_C$ is isomorphic to the Hirzebruch surface $\Sigma_4$).

\begin{proposition}
\label{prop:YC}
Denote by $Y_C$ the closed subscheme $\Sigma_C\cup C^{[2]}$ in $X^{[2]}$. 
Then $\Phi(\OO_{Y_C}\otimes \LL)\cong \OO_C\neglin{2}$.  
\end{proposition}

Before we prove the proposition, we need to recall another feature of the Hilbert scheme of two points, namely
that the functor 
\[
\Theta \coloneqq \iota_*\rho^* \colon \Db(X)\to \Db(X^{[2]})
\]
is spherical; see \cite[Rem.\ 4.3]{Krug-autos} or \cite[Thm 4.26(ii)]{Krug-Ploog-Sosna}. This means, in particular, that the associated twist functor $\TTT_{\Theta}$, defined by the triangle of functors
\[
\Theta\Theta^R \xrightarrow{\eps} \id  \to \TTT_{\Theta}\,,
\]
where $\eps$ is the counit of adjunction, is an autoequivalence of $\Db(X^{[n]})$.
The right adjoint $\Theta^R$ of $\Theta$ is given by
\begin{equation}
\tag{$\ast$}
\label{eq:right-adjoint}
\Theta^R\cong \rho_* \iota^!\cong \rho_*\bigl(\iota^*(\blank)\otimes \OO_E(E)\bigr)[-1]\,.
\end{equation}

\begin{lemma}
The subvarieties $C^{[2]}$ and $E$ of $X^{[2]}$ intersect transversally and $\rho$ maps the scheme-theoretic intersection $E \cap C^{[2]}$ isomorphically to $\delta(C)\subset X^{(2)}$.
\end{lemma}
\begin{proof}
The second assertion implies the first one since transversality means that the scheme-theoretic intersection $E\cap C^{[2]}$ is reduced and of the expected dimension $1$. 

The composition $C^{[2]}\hookrightarrow X^{[2]}\xrightarrow \rho X^{(2)}$ is a closed embedding with image $C^{(2)}\subset X^{(2)}$. Using the right cartesian diagram in \eqref{diag:E}, it follows that $\rho$ maps $E\cap C^{[2]}$ isomorphically to the scheme-theoretic intersection $\delta(X)\cap C^{(2)}$. 

Hence, we only have to prove that $\delta(X)\cap C^{(2)}=\delta(C)$. This question is local in the analytic topology so that we may assume that $X=\Spec \IC[x_1,x_2]$ and $C=\Spec \IC[x_1]$. We set $s_i=x_i+y_i$ and $t_i=x_i-y_i$ so that 
$X^2= \Spec\IC[s_1,s_2,t_1,t_2]$ with the natural action of $\sym_2=\langle \tau\rangle$ given by $\tau\cdot s_i=s_i$ and $\tau\cdot t_i=-t_i$. Therefore,
\[
 \reg(X^{(2)})=\IC[s_1,s_2,t_1,t_2]^{\sym_2}=\IC[s_1,s_2,t_1^2, t_1t_2,t_2^2]\,.
\]
The ideal of $\delta(X)\subset X^{(2)}$ is given by 
$I=(t_1,t_2)^{\sym_2}=(t_1^2, t_1t_2,t_2^2)$
and the ideal of $C^{(2)}\subset X^{(2)}$ is given by 
$J=(s_2,t_2)^{\sym_2}=(s_2, t_2^2)$.
Hence, \[\delta(X)\cap C^{(2)}= \Spec\bigl( \IC[s_1,s_2,t_1^2, t_1t_2,t_2^2]/(I+J)\bigr) =\Spec \IC[s_1]=\delta(C)\,.\qedhere\]   
\end{proof}

\begin{lemma}
\label{lem:OCtwist}
For the spherical twist $\TTT_\Theta$, we have $\TTT_\Theta(\OO_{C^{[2]}}(-E))\cong \OO_{Y_C}$. 
\end{lemma}

\begin{proof}
By $\reg_{C^{[2]}}(-E)$ we mean $\left.\reg_{X^{[2]}}(-E)\right|_{C^{[2]}}$.
Using \eqref{eq:right-adjoint} and the previous lemma, we compute
\[
\Theta^R(\OO_{C^{[2]}}(-E))\cong  \OO_C[-1]\,.
\]
Note that $\Theta(\OO_C)\cong \OO_{\Sigma_C}$, so the twist triangle applied to $\OO_{C^{[2]}}(-E)$ becomes, after shift,
\[
\OO_{C^{[2]}}(-E)\to \TTT_\Theta(\OO_{C^{[2]}}(-E))\to \OO_{\Sigma_C}\,.
\]
The long exact cohomology sequence shows that $\cH^i\bigl(\TTT_\Theta(\OO_{C^{[2]}}(-E))\bigr)=0$ for $i\neq 0$. Hence, the triangle reduces to the short exact sequence
\begin{align}\label{eq:es1}
0\to \OO_{C^{[2]}}(-E\cap C^{[2]})\to \cH^0\bigl(\TTT_\Theta(\OO_{C^{[2]}}(-E))\bigr)\to \OO_{\Sigma_C}\to 0.
\end{align}
As $\Sigma_C \cap C^{[2]} = E \cap C^{[2]}$, there is the canonical short exact sequence 
\begin{align}\label{eq:es2}
0\to \OO_{C^{[2]}}(-E\cap C^{[2]})\to \reg_{Y_C}\to \OO_{\Sigma_C}\to 0.
\end{align}
One can compute that $\Ext^1(\reg_{\Sigma_C},\OO_{C^{[2]}}(-\Sigma_C\cap C^{[2]}))=\IC$, using for example \cite[Thm.\ A.1]{CKS-ext}.
It follows that \eqref{eq:es1} and \eqref{eq:es2} coincide, so
\[
\TTT_\Theta(\OO_{C^{[2]}}(-E))\cong \cH^0\bigl(\TTT_\Theta(\OO_{C^{[2]}}(-E))\bigr)\cong\OO_{Y_C}\,.\qedhere
\]
\end{proof}

\begin{proof}[Proof of \autoref{prop:YC}]
Combining the formulae of \cite[Thm.\ 4.26]{Krug-Ploog-Sosna}, we get an isomorphism of functors 
\[
\Phi^{-1}(\Phi(\blank)\otimes \alt)\cong \LL\otimes \TTT_\Theta(\blank\otimes \LL^{-2})\,,
\]
where $\LL^2 = \OO(E)$.
Combining this with 
\autoref{prop:CMcKay} gives
\[
 \Phi^{-1}(\OO_C\neglin{2})\cong \LL\otimes \TTT_\Theta(\OO_{C^{[2]}}\otimes \LL^{-2})\,.
\]
Now the assertion follows by \autoref{lem:OCtwist}.
\end{proof}

For $C_i, C_j\in X$ two $(-2)$-curves which intersect in one point, $Y_{C_i}$ and $C_j^{[2]}$ intersect transversally in one point of $X^{[2]}$. This confirms geometrically that
\[
 \Hom^*_{\Db_{\sym_2}(X^2)}(\OO_{C_i}\neglin{2},\OO_{C_j}\poslin{2})\cong \Hom^*_{\Db(X^{[2]})}(\OO_{Y_{C_i}}\otimes\LL,\OO_{C_j^{[2]}})=\IC[-2]\,;   
\]
compare \autoref{Prop:gradedhom}.

\addtocontents{toc}{\protect\setcounter{tocdepth}{-1}}  

\bigskip
\begin{tabbing}
\emph{Contact:} \=\verb+andreas.hochenegger@unimi.it+ \\
                \>\verb+andkrug@mathematik.uni-marburg.de+
\end{tabbing}

\end{document}